\documentclass[12pt]{article}

\usepackage[margin=1in]{geometry}  
\usepackage{amsmath,amsfonts,amsthm,amssymb}	
\usepackage{tikz}							
\usepackage[shortlabels]{enumitem}		
\usepackage{hyperref}
\usetikzlibrary{calc}

\theoremstyle{plain}
\newtheorem{thm}{Theorem}[section]
\newtheorem{lem}[thm]{Lemma}
\newtheorem{prop}[thm]{Proposition}
\newtheorem{cor}[thm]{Corollary}

\theoremstyle{definition}
\newtheorem{df}[thm]{Definition}
\newtheorem{exmp}[thm]{Example}

\newcommand{\comment}[1]{}

\newcommand{\bbR}{\mathbb{R}}
\newcommand{\bbZ}{\mathbb{Z}}

\newcommand{\abs}[1]{\left|#1\right|}
\newcommand{\D}{\Delta}
\newcommand{\emp}{\emptyset}
\newcommand{\eq}[1]{\left[#1\right]}

\newcommand{\gr}[1]{\mathsf{#1}}

\newcommand{\mt}{\text{ mod }2}
\newcommand{\Par}[1]{\left(#1\right)}

\renewcommand{\phi}{\varphi}
\newcommand{\Rt}{\mathbb{R}^2}

\newcommand{\sd}{\textnormal{sd}}
\newcommand{\set}[1]{\left\{#1\right\} }
\newcommand{\su}{\subseteq}

\newcommand{\vk}[1]{\mathfrak{#1}}
\newcommand{\Zt}{\mathbb{Z}_2}

\comment{
barycentric
cohomological
cohomology
embeddability
equivariant
Hanani
Heawood
homeomorphic
homeomorphism
homotopy
Kampen
Kuratowski's
piecewise
planarity
representability
representable
satisfiability
satisfiable
simplicial
Sinden
Tutte
}

\numberwithin{equation}{section}

\newcounter{figure_num}

\begin{document}

\title{A Van Kampen type obstruction for string graphs}
\author{Moshe J. White \thanks{Research supported in part by the Israel Science Foundation (Grant No. 2669/21) and in part by the ERC Advanced Grant (Grant No. 834735)} \\
Institute of Mathematics,\\
Hebrew University, Jerusalem, Israel}
\date{}
\maketitle

\begin{abstract}
In this paper we prove that a graph is a string graph (the intersection graph of curves in the plane) if and only if it admits a drawing in the plane with certain properties. This also allows us to define an algebraic obstruction, similar to the Van Kampen obstruction to embeddability, which must vanish for every string graph. However, unlike in the case of graph planarity this obstruction is incomplete.


\end{abstract}

\section{Introduction} 

A graph $\gr G=(V,E)$ is a \emph{string graph}, if it is the intersection graph of a collection of curves in the plane: we may find $\set{\gamma_v\ :\ v\in V}$ such that every $\gamma_v$ is the image of a continuous map $[0,1]\to\Rt$, and for any two vertices $v\neq w$ we have $\gamma_v\cap\gamma_w\neq\emp$ if and only if $\set{v,w}\in E$. The collection of curves $\set{\gamma_v}$ is said to be a representation of $\gr G$. For a recent brief exposition of string graphs, see \cite[\S 1]{PRY}.

In some sense, we consider string graphs as a generalization of planar graphs. In one of the earliest papers dealing with string graphs, Sinden \cite{Sin} already determined that every planar graph is a string graph, while some non-planar graphs are string graphs (such as any complete graph) and some are not (such as a barycentric subdivision of any non-planar graph). The approach presented in this paper arose from previous work of the author \cite{Whi}, dealing with simplicial complexes representable by convex sets in $\bbR^d$, where representability was described in terms of the existence of certain linear maps. Furthermore, necessary conditions (but usually not sufficient) for representability were given. The lowest dimensions in which representability is not fully classified is for graphs ($1$-dimensional simplicial complexes) representable in dimension $2$. The topological restrictions on graphs are in fact satisfied by all string graphs (a class strictly containing graphs representable in dimension $2$, see \cite[Example 4.13]{Whi}).

In this paper we further explore the necessary topological restrictions on string graphs. Keeping in mind the analogy between string graphs and planar graphs, a drawing of $\gr G$ in the plane where \emph{none} of the edges intersect indicates that $\gr G$ is planar. If we instead allow \emph{some} intersections among the edges of $\gr G$, the drawing indicates that our graph is a string graph. The precise conditions classifying string graphs (in terms of edge intersections) is given in Theorem \ref{ob_2_eq}. Continuing our analogy, we consider the (strong) Hanani-Tutte Theorem (see \cite{Sch} for more details):
\begin{thm}[Hanani-Tutte] \label{HT_thm}
A graph $\gr G$ can be drawn in $\Rt$ such that every pair of disjoint edges cross an even number of times, if and only if $\gr G$ is planar.
\end{thm}
The ``if'' part is clear, as an embedding of $\gr G$ admits no crossings among edges. The converse may be deduced from Kuratowski's theorem \cite{Tut}, or even proved directly \cite{Sar}. The even crossing condition is equivalent to the existence of a solution to a system of linear equations over $\Zt$, described in cohomological terms as the vanishing of the (mod 2) Van Kampen obstruction. We introduce a modified Van Kampen obstruction that must vanish for every string graph $\gr G$ (Theorem \ref{ob_1_van}). This theorem essentially states that if $\gr G$ is a string graph, then there exists a drawing of $\gr G$ in the plane such that certain pairs of edges cross an even number of times. Unlike with graph planarity, the converse fails, as seen in Example \ref{Kratochvil_vanishing}.

\subsection{Notation for graphs}

A graph $\gr G=(V,E)$ is an abstract collection of vertices $V=V(\gr G)$ and edges $E=E(\gr G)$, where each edge is a set containing exactly two vertices. 

Abusing notation, when we use a graph as the domain or range of a function, such as $f:\gr G\to\Rt$, this $\gr G$ is to be understood as the topological space which is the geometric realization of the abstract graph $\gr G$ (a point for each vertex and a line segment for every edge, attached at the endpoints to the points corresponding to its vertices). We sometimes identify an edge $\alpha\in E$ with the corresponding line segment, so $f(\alpha)$ denotes the image of the line segment corresponding to $\alpha$. We also use the abbreviated notation $vw$ in place of $\set{v,w}$ (often an edge).

We define \emph{a subdivision} of $\gr G$ to be a graph $\gr H$ together with a homeomorphism $\gr G\to\gr H$ taking vertices of $\gr G$ to vertices of $\gr H$. If no edge of $\gr G$ is mapped to a single edge of $\gr H$, we say that $\gr H$ is a \emph{proper} subdivision. We let $\gr G^*$ denote the barycentric subdivision of $\gr G=(V,E)$:
$$V(\gr G^*)=V\cup E\quad,\quad E(\gr G^*)=\set{v\alpha\ :\ v\in\alpha\in E}.$$

We usually require that maps $f:\gr G\to\Rt$ are piecewise linear and in general position: there exists $\gr H$ which is a subdivision of $\gr G$, such that the composition of $f$ with the (inverse) homeomorphism $\gr H\to\gr G$ is linear (every edge is mapped to a line segment), and takes no three vertices of $\gr H$ to the same line. 

\section{Modified Van Kampen obstructions} 

We adopt the geometric definition of the (mod $2$) Van Kampen obstruction for graph planarity from \cite{Sar}, and also borrow notation from \cite[Appendix D]{MTW}. For a graph $\gr G=(V,E)$, we let


$$P_\D(\gr G)=\set{(\alpha,\beta)\in E^2\ :\ \alpha\cap\beta=\emp}$$
denote the collection of disjoint pairs of edges in $\gr G$. The Van Kampen obstruction can be defined as an equivariant cohomology class (with coefficients in $\bbZ/2$) of the deleted product
$$\gr G^2_P:=\bigcup_{(\alpha,\beta)\in P} F(\alpha)\times F(\beta),$$
where $P=P_\D(\gr G)$ and $F:\gr G\to\bbR^3$ is some embedding. The only change we introduce to the ordinary mod 2 Van Kampen obstruction, is that we allow for $P\su P_\D$ to be some symmetric subset. We do not assume familiarity with cohomology or the Van Kampen obstruction, readers familiar with the argument (using finger moves) that the obstruction must vanish for planar graphs may skip to Proposition \ref{vk_prop}, which is proved analogously to the more familiar case of the deleted product.

The set $P$ should be thought of as a collection of edge pairs whose intersections are forbidden: our goal will be to find some map $f:\gr G\to\Rt$ such that $f(\alpha)\cap f(\beta)=\emp$ for every $(\alpha,\beta)\in P$ (or prove that no such $f$ exists).

Since any pair $(\alpha,\beta)\in P$ involves disjoint edges, the intersection $f(\alpha)\cap f(\beta)$ is either empty or consists of finitely many points, all in the interior of both edges. Assuming $f$ is piecewise linear and in general position, we define a vector $o_f\in\Zt^P$ by
$$o_f(\alpha,\beta)=\left|f(\alpha)\cap f(\beta)\right|\mt.$$
If $f$ is not piecewise linear and in general position, we say that $o_f$ is undefined. Our next step is to define vectors corresponding to the Van Kampen ``finger moves''. For every vertex $u\in V$ and edge $\omega\in E$ such that $u\notin\omega$, define $\phi_{\omega,u}\in\Zt^P$ by

$$\phi_{\omega,u}\Par{\alpha,\beta}=\begin{cases}
1 & \text{if $\alpha=\omega$ and $u\in\beta$,}\\
1 & \text{if $\beta=\omega$ and $u\in\alpha$,}\\
0 & \text{otherwise}.
\end{cases}$$

The collection of all finger move vectors $\Phi=\set{\phi_{\omega,u}:u\in V,\ \omega\in E,\ u\notin \omega}$ defines an equivalence relation on $\Zt^P$, where two vectors $\phi,\psi$ are equivalent if $\phi-\psi\in\textnormal{span}_{\Zt}(\Phi)$.

Note that both $o_f$ and $\phi_{\omega,u}$ in fact depend on the set of pairs $P$, and we will denote them by $o_f^P$ and $\phi_{\omega,u}^P$ respectively if $P$ isn't clear from the context. We can now choose some function $f:\gr G\to\Rt$, and define:

\begin{df} \label{original_vk}
The Van Kampen obstruction of a graph $\gr G$ is defined as the equivalence class ${\vk o_{\gr G}=\eq{o_f^{P_\D(\gr G)}}}$.
\end{df}
And more generally:
\begin{df} \label{vk_def}
Let $\gr G=(V,E)$ be a graph, and let $P\su E^2$ be a symmetric collection of disjoint edges: any $(\alpha,\beta)\in P$ satisfies $(\beta,\alpha)\in P$ and $\alpha\cap\beta=\emp$.\\
The Van Kampen obstruction of $\gr G$ (modified for $P$) is the equivalence class
$$\vk o_{\gr G}(P)=\eq{o_f^P}.$$
We say that the obstruction \emph{vanishes} if $\vk o_{\gr G}(P)=\eq0$.
\end{df}

The definition implies that $\vk o_{\gr G}(P)$ doesn't depend on our choice of $f$. Indeed, the vectors equivalent to \emph{any} $o_f$ are precisely \emph{all} vectors of the form $o_g$. Explicitly:

\begin{lem} \label{vk_lem}
For $\gr G, P$ as above and $f:\gr G\to\Rt$, we have:
\begin{enumerate}
\item For every $\phi_{\omega,u}\in\Phi$, there exists $g:\gr G\to\Rt$, such that $o_g=o_f+\phi_{\omega,u}$,
\item For every $g:\gr G\to\Rt$, we have $o_f-o_g\in\textnormal{span}_{\Zt}(\Phi)$.
\end{enumerate}
\end{lem}

\begin{proof}
We achieve (1) by applying a finger move to $f$: take a thin tube around a path connecting an interior point of $f(\omega)$ to the vertex $u$, such that $u$ is the only vertex within the tube. Then we only modify the image of $\omega$ to go around the boundary of the tube. The resulting modified function $g$ has $o_g=o_f+\phi_{\omega,u}$ (see Figure \ref{finger_move_figure}).

\begin{center}
\refstepcounter{figure_num}\label{finger_move_figure}
\begin{tikzpicture}[scale=0.7]
\coordinate(w0) at (0,0);
\filldraw (w0) circle (2.86pt) node[left] {$1$};
\coordinate(w1) at (0,6);
\filldraw (w1) circle (2.86pt) node[left] {$2$};
\coordinate(v0) at (3.94,6);
\filldraw (v0) circle (2.86pt) node[right] {$3$};
\coordinate(v1) at (4.88,2.55);
\filldraw (v1) circle (2.86pt) node[above right] {$4$};
\coordinate(v2) at (6.87,1.62);
\filldraw (v2) circle (2.86pt) node[below] {$5$};
\coordinate(v3) at (2.32,0);
\filldraw (v3) circle (2.86pt) node[left] {$6$};

\draw (w0) -- (w1);
\draw (w1) -- (v1);
\draw (v0) -- (v1);
\draw (v2) -- (v1);
\draw (v3) -- (v1);
\draw (v0) -- (v3);

\begin{scope}[shift={(11.5,0)}]
\coordinate(w0) at (0,0);
\filldraw (w0) circle (2.86pt) node[left] {$1$};
\coordinate(w1) at (0,6);
\filldraw (w1) circle (2.86pt) node[left] {$2$};
\coordinate(v0) at (3.94,6);
\filldraw (v0) circle (2.86pt) node[right] {$3$};
\coordinate(v1) at (4.88,2.55);
\filldraw (v1) circle (2.86pt) node[above right] {\scriptsize $4$};
\coordinate(v2) at (6.87,1.62);
\filldraw (v2) circle (2.86pt) node[below] {$5$};
\coordinate(v3) at (2.32,0);
\filldraw (v3) circle (2.86pt) node[left] {$6$};

\draw[thick] (w0) -- (0,2.0) to [out=90,in=180] (0.4,2.4) -- (4.2,2.4) arc (192.46 : 527.54 : 0.7) -- (0.4,2.7) to [out=180,in=270] (0,3.1) -- (w1);

\draw (w1) -- (v1);
\draw (v0) -- (v1);
\draw (v2) -- (v1);
\draw (v3) -- (v1);
\draw (v0) -- (v3);



\end{scope}
\end{tikzpicture}
\footnotesize
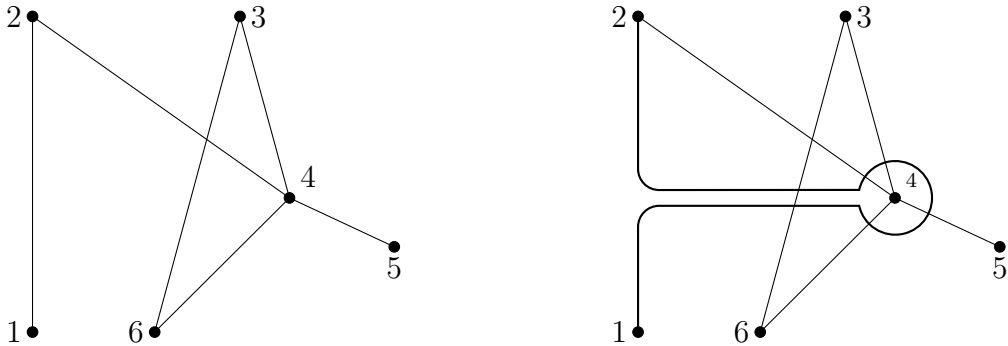
\\Figure \arabic{figure_num}: a map obtained from a linear $f:\gr G\to\Rt$ (left) by applying a finger move taking $\omega=12$ around $u=4$ (right). Note that we ignore intersections of $\omega$ with edges $\beta$ satisfying $(\omega,\beta)\notin P$ (such as $\beta=24$ in this example, as follows from $P\su P_\D$). Self intersections of $\omega$ may also be created and ignored for the same reason. Intersections of $\omega$ with edges $\beta$ such that $u\notin\beta\neq\omega$ (such as $\beta=36$ above) are created in pairs and thus have no effect on $o_f$.
\end{center}

For (2), we say that a homotopy $h_t:\gr G\to\Rt$ is generic, if $o_{h_t}$ is defined for all but finitely many $t\in[0,1]$. In this case, as $t$ increases from $0$ to $1$, the only occasion where $o_{h_t}$ is modified is when a vertex $u$ passes through an edge $\omega$, and this modifies $o_{h_t}$ by adding $\phi_{\omega,u}$. For example, if we take $f=h_0$ to be a linear map of the graph with vertices positioned as in Figure \ref{finger_move_figure} (left), and as $t$ increases modify our function by gradually pulling vertex $4$ to the left, such that at time $t_1>0$ the vertex passes through the edge $36$, and at time $t_2\in(t_1,1)$ the vertex passes through the edge $12$. Then we have $o_{h_t}=o_f+\phi_{36,4}$ for $t\in(t_1,t_2)$, and $o_{h_t}=o_f+\phi_{36,4}+\phi_{12,4}$ for $t>t_2$. 
Given any two (general position piecewise linear) functions $f,g:\gr G\to\Rt$, we can easily find such a generic homotopy between them, completing our proof.
\end{proof}

This clearly implies that $\vk o_{\gr G}$ vanishes if $\gr G$ is planar. For modified obstructions, we deduce from the lemma:

\begin{prop} \label{vk_prop}
Let $\gr G=(V,E)$ be a graph, $P\su E^2$ a symmetric collection of disjoint edges. Then $\vk o_{\gr G}(P)$ vanishes if and only if there exists $f:\gr G\to\Rt$, such that $\left|f(\alpha)\cap f(\beta)\right|$ is even for every $(\alpha,\beta)\in P$.
\end{prop}

\section{The string graph obstruction} 

We are now ready to define the string obstruction of a graph:
\begin{df} \label{first_ob}
For a graph $\gr G=(V,E)$, denote
$$P=P_s(\gr G)=\set{(\alpha,\beta)\in E^2\ :\ vw\notin E\text{ for every } v\in\alpha,w\in\beta}$$
and define the \emph{string obstruction} of $\gr G$ as 
$$\vk o^s_{\gr G}=\vk o_{\gr G}(P).$$
\end{df}

We claim:

\begin{thm} \label{ob_1_van}
Let $\gr G$ be a string graph. Then $\vk o^s_\gr{G}$ vanishes.
\end{thm}

The proof will be postponed until after establishing Lemma \ref{ob_eq}, and first we consider some examples:

\begin{cor}
The Heawood graph (see Figure \ref{heawood_figure} below) is not a string graph.
\end{cor}

\begin{center}
\refstepcounter{figure_num}\label{heawood_figure}
\begin{tikzpicture}

\coordinate (v0) at (360 * 0 / 14:2);
\filldraw (v0) circle (2pt);
\coordinate (v1) at (360 * 1 / 14:2);
\filldraw (v1) circle (2pt);
\coordinate (v2) at (360 * 2 / 14:2);
\filldraw (v2) circle (2pt);
\coordinate (v3) at (360 * 3 / 14:2);
\filldraw (v3) circle (2pt);
\coordinate (v4) at (360 * 4 / 14:2);
\filldraw (v4) circle (2pt);
\coordinate (v5) at (360 * 5 / 14:2);
\filldraw (v5) circle (2pt);
\coordinate (v6) at (360 * 6 / 14:2);
\filldraw (v6) circle (2pt);
\coordinate (v7) at (360 * 7 / 14:2);
\filldraw (v7) circle (2pt);
\coordinate (v8) at (360 * 8 / 14:2);
\filldraw (v8) circle (2pt);
\coordinate (v9) at (360 * 9 / 14:2);
\filldraw (v9) circle (2pt);
\coordinate (v10) at (360 * 10 / 14:2);
\filldraw (v10) circle (2pt);
\coordinate (v11) at (360 * 11 / 14:2);
\filldraw (v11) circle (2pt);
\coordinate (v12) at (360 * 12 / 14:2);
\filldraw (v12) circle (2pt);
\coordinate (v13) at (360 * 13 / 14:2);
\filldraw (v13) circle (2pt);

\node at (360 * 0 / 14:2.3) {1};
\node at (360 * 1 / 14:2.3) {2};
\node at (360 * 2 / 14:2.3) {3};
\node at (360 * 3 / 14:2.3) {4};
\node at (360 * 4 / 14:2.3) {5};
\node at (360 * 5 / 14:2.3) {6};
\node at (360 * 6 / 14:2.3) {7};
\node at (360 * 7 / 14:2.3) {8};
\node at (360 * 8 / 14:2.3) {9};
\node at (360 * 9 / 14:2.3) {a};
\node at (360 * 10 / 14:2.3) {b};
\node at (360 * 11 / 14:2.3) {c};
\node at (360 * 12 / 14:2.3) {d};
\node at (360 * 13 / 14:2.3) {e};

\draw (v0) -- (v1) -- (v2) -- (v3) -- (v4) -- (v5) -- (v6) -- (v7) -- (v8) -- (v9) -- (v10) -- (v11) -- (v12) -- (v13) -- (v0);

\draw (v1) -- (v6);
\draw (v2) -- (v11);
\draw (v3) -- (v8);
\draw (v4) -- (v13);
\draw (v5) -- (v10);
\draw (v7) -- (v12);
\draw (v9) -- (v0);

\end{tikzpicture}
\footnotesize\\
Figure \arabic{figure_num}: the Heawood graph.
\end{center}

\begin{proof}
Let $\gr G$ denote the Heawood graph and $P=P_s(\gr G)$. 
Denote
$$P_\text{set}=\set{\set{\alpha,\beta}\ :\ (\alpha,\beta)\in P},$$
which is well defined because $P$ is symmetric. Also note that if $\psi\in\Zt^P$ is either a finger move vector or some $o_f$, then $\psi(\alpha,\beta)=\psi(\beta,\alpha)$ for every $(\alpha,\beta)\in P$. We may therefore say that such $\psi$ is \emph{even}, if
$$\sum_{(\alpha,\beta)\in P_\text{set}}\psi(\alpha,\beta)=0\in\Zt,$$
and otherwise say that $\psi$ is \emph{odd}. Furthermore, for any $f:\gr G\to\Rt$ we have
$$\sum_{(\alpha,\beta)\in P_\text{set}} o_f(\alpha,\beta)=\sum_{(\alpha,\beta)\in P_\text{set}}\left|f(\alpha)\cap f(\beta)\right|\mt.$$
For the function $f$ displayed in Figure \ref{heawood_figure}, we obtain that $o_f$ is odd: there are $7$ crossings in total among edge pairs belonging to $P_\text{set}$ - those which are closest to the center of the figure, such as $f(1a)\cap f(3c)$. The other $7$ crossings such as $f(1a)\cap f(5e)$ involve pairs of edges not belonging to $P_\text{set}$.

Our claim will now follow by proving that every finger move for $\gr G$ is even: this implies that $o_g$ is odd for every $g:\gr G\to\Rt$, therefore $\vk o_{\gr G}^s$ doesn't vanish and Theorem \ref{ob_1_van} implies that $\gr G$ is not a string graph.

Consider a finger move vector $\phi_{\omega,u}$, where $\omega=12$ and $u\notin\omega$. If $\phi_{\omega,u}(\alpha,\beta)=1$ then one of the edges $\alpha,\beta$ equals $\omega$ and the other edge contains $u$ (w.l.o.g. $\alpha=\omega$ and $u\in\beta$). Furthermore, $(\omega,\beta)\in P$ only for edges $\beta$ whose vertices are not adjacent to either of the vertices of $\omega$, in this case the vertices of $\beta$ must be two of $\set{4,5,6,8,9,b,c,d}$. These vertices form a cycle of length $8$ in $\gr G$, in the cyclic order $456bcd89$. We may now deduce that $\phi_{\omega,u}$ must be even: if $u\in\set{3,7,a,e}$ then $\phi_{\omega,u}$ is the zero vector. Otherwise $u\in\set{4,5,6,8,9,b,c,d}$ and there are exactly two vertices $v$ for which $\beta=uv$ has $\phi_{\omega,u}(\omega,\beta)=1$ (those $v$ which are adjacent to $u$ in the cycle), and again $\phi_{\omega,u}$ is even.

For other edges $\omega$ the proof is analogous: if $\omega=1a$ then non-zero values are obtained when $u$ and $\beta$ are part of the cycle $345678dc$, and if $\omega=1e$ then non-zero values are obtained when $u$ and $\beta$ are part of the cycle $349876bc$. The proof for any other $\omega$ easily follows from the previous cases by rotations and reflections fixing Figure \ref{heawood_figure}.
\end{proof}

There are many more cases of graphs for which the string obstruction doesn't vanish, and Theorem \ref{ob_1_van} gives us a unified approach in proving that they aren't string graphs. For example, if $\gr H$ is a proper barycentric subdivision of a non-planar graph $\gr G$, the non-vanishing of $\vk o_{\gr H}^s$ follows from Hananni-Tutte (Theorem \ref{HT_thm}): every map $f:\gr H\to\Rt$ determines a map $\gr G\to\Rt$, so there are two disjoint $\gr G$-edges which intersect an odd number of times. In $\gr H$ these are divided into pairs of edges in $P_s(\gr G)$, so at least one such pair must intersect an odd number of times, thus $o_f$ is non-zero.

In \cite{GKK}, Goljan, Kratochv\'il, and Ku\v cera explored string graphs and gave a number of examples of small non-string graphs, including two with $12$ vertices \cite[Example 3.9]{GKK}, also proving that this is the minimal number of vertices in any string graphs. We won't include all such examples and proofs - for any graph $\gr G$, determining if the string graph obstruction vanishes can be done by solving a linear system of equations over $\Zt$: $\vk o_{\gr G}^s$ vanishes if and only if $o_f\in\text{span}_{\Zt}(\Phi)$ for some $f:\gr G\to\Rt$. This paper also provides an example in which the converse of Theorem \ref{ob_1_van} doesn't hold:

\begin{exmp} \label{Kratochvil_vanishing}
Consider the graph GP defined \cite{GKK}, and an important infinite class of non-string graphs denoted $C*\overline C_n$ for every $n\geq 5$ \cite{Kra1}. The two functions $f:\gr G\to\Rt$ displayed in Figure \ref{Kratochvil_figure} for $\gr G=GP$ and $\gr G=C*\overline C_6$ have $o_f=0$, as all crossings involve pairs of edges not in $P_s(\gr G)$.

\begin{center}
\refstepcounter{figure_num}\label{Kratochvil_figure}
\begin{tikzpicture}

\coordinate (v1) at (90 - 72 * 0:1.4);
\filldraw (v1) circle (2pt);
\coordinate (v2) at (90 - 72 * 1:1.4);
\filldraw (v2) circle (2pt);
\coordinate (v3) at (90 - 72 * 2:1.4);
\filldraw (v3) circle (2pt);
\coordinate (v4) at (90 - 72 * 3:1.4);
\filldraw (v4) circle (2pt);
\coordinate (v5) at (90 - 72 * 4:1.4);
\filldraw (v5) circle (2pt);

\coordinate (v6) at (90 - 72 * 0:2.5);
\filldraw (v6) circle (2pt);
\coordinate (v7) at (90 - 72 * 1:2.5);
\filldraw (v7) circle (2pt);
\coordinate (v8) at (90 - 72 * 2:2.5);
\filldraw (v8) circle (2pt);
\coordinate (v9) at (90 - 72 * 3:2.5);
\filldraw (v9) circle (2pt);
\coordinate (v10) at (90 - 72 * 4:2.5);
\filldraw (v10) circle (2pt);

\coordinate (v11) at ($(v6)!0.5!(v7)$);
\filldraw (v11) circle (2pt);
\coordinate (v12) at ($(v7)!0.5!(v8)$);
\filldraw (v12) circle (2pt);
\coordinate (v13) at ($(v8)!0.5!(v9)$);
\filldraw (v13) circle (2pt);
\coordinate (v14) at ($(v9)!0.5!(v10)$);
\filldraw (v14) circle (2pt);
\coordinate (v15) at ($(v6)!0.5!(v10)$);
\filldraw (v15) circle (2pt);
\coordinate (v16) at ($(v2)!0.5!(v7)$);
\filldraw (v16) circle (2pt);
\coordinate (v17) at ($(v4)!0.5!(v9)$);
\filldraw (v17) circle (2pt);
\coordinate (v18) at ($(v5)!0.5!(v10)$);
\filldraw (v18) circle (2pt);

\draw (v1) -- (v3) -- (v5) -- (v2) -- (v4) -- (v1);
\draw (v6) -- (v7) -- (v8) -- (v9) -- (v10) -- (v6);
\draw (v1) -- (v6);
\draw (v2) -- (v7);
\draw (v3) -- (v8);
\draw (v4) -- (v9);
\draw (v5) -- (v10);

\begin{scope}[shift={(7,0)}]
\coordinate (v1) at (120 - 60 * 0:1.4);
\filldraw (v1) circle (2pt);
\coordinate (w1) at (120 - 60 * 0:2.5);
\filldraw (w1) circle (2pt);
\coordinate (v2) at (120 - 60 * 1:1.4);
\filldraw (v2) circle (2pt);
\coordinate (w2) at (120 - 60 * 1:2.5);
\filldraw (w2) circle (2pt);
\coordinate (v3) at (120 - 60 * 2:1.4);
\filldraw (v3) circle (2pt);
\coordinate (w3) at (120 - 60 * 2:2.5);
\filldraw (w3) circle (2pt);
\coordinate (v4) at (120 - 60 * 3:1.4);
\filldraw (v4) circle (2pt);
\coordinate (w4) at (120 - 60 * 3:2.5);
\filldraw (w4) circle (2pt);
\coordinate (v5) at (120 - 60 * 4:1.4);
\filldraw (v5) circle (2pt);
\coordinate (w5) at (120 - 60 * 4:2.5);
\filldraw (w5) circle (2pt);
\coordinate (v6) at (120 - 60 * 5:1.4);
\filldraw (v6) circle (2pt);
\coordinate (w6) at (120 - 60 * 5:2.5);
\filldraw (w6) circle (2pt);

\coordinate (z1) at ($(w1)!0.5!(w2)$);
\filldraw (z1) circle (2pt);
\coordinate (p1) at ($(v1)!0.5!(w1)$);
\filldraw (p1) circle (2pt);
\coordinate (z2) at ($(w2)!0.5!(w3)$);
\filldraw (z2) circle (2pt);
\coordinate (p2) at ($(v2)!0.5!(w2)$);
\filldraw (p2) circle (2pt);
\coordinate (z3) at ($(w3)!0.5!(w4)$);
\filldraw (z3) circle (2pt);
\coordinate (p3) at ($(v3)!0.5!(w3)$);
\filldraw (p3) circle (2pt);
\coordinate (z4) at ($(w4)!0.5!(w5)$);
\filldraw (z4) circle (2pt);
\coordinate (p4) at ($(v4)!0.5!(w4)$);
\filldraw (p4) circle (2pt);
\coordinate (z5) at ($(w5)!0.5!(w6)$);
\filldraw (z5) circle (2pt);
\coordinate (p5) at ($(v5)!0.5!(w5)$);
\filldraw (p5) circle (2pt);
\coordinate (z6) at ($(w6)!0.5!(w1)$);
\filldraw (z6) circle (2pt);
\coordinate (p6) at ($(v6)!0.5!(w6)$);
\filldraw (p6) circle (2pt);

\draw (v1) -- (v3);
\draw (v1) -- (v5);
\draw (v2) -- (v4);
\draw (v2) -- (v6);
\draw (v3) -- (v5);
\draw (v4) -- (v6);
\draw (v1) to [bend left=15] (v4);
\draw (v3) to [bend left=15] (v6);
\draw (v5) to [bend left=15] (v2);

\draw (w1) -- (w2) -- (w3) -- (w4) -- (w5) -- (w6) -- (w1);
\draw (v1) -- (w1);
\draw (v2) -- (w2);
\draw (v3) -- (w3);
\draw (v4) -- (w4);
\draw (v5) -- (w5);
\draw (v6) -- (w6);
\end{scope}
\end{tikzpicture}
\footnotesize\\
\medskip
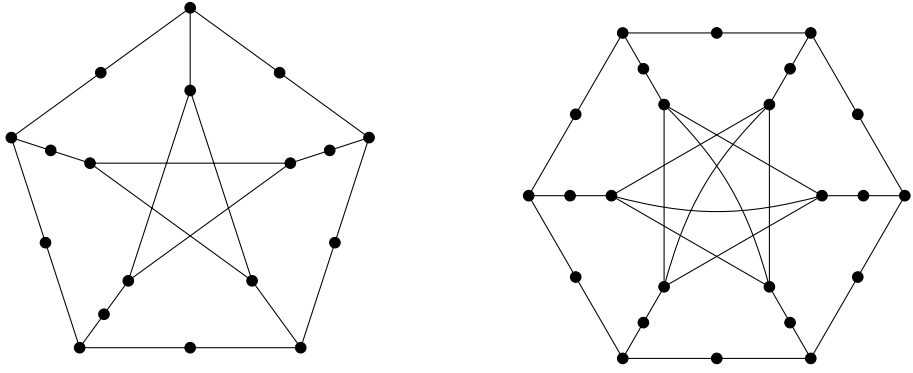
Figure \arabic{figure_num}: the graphs $GP$ (left) and $C*\overline C_6$ (right).
\end{center}

A similar $C*\overline C_n\to\Rt$ exists for every $n\geq 5$, proving that the string obstruction fails to classify infinitely many non-string graphs.
\end{exmp}

\subsection{The obstruction via barycentric subdivisions}

\comment{
\begin{df}
$$P_s(\gr G)=\set{(\alpha,\beta)\in E^2\ :\ vw\notin E\text{ for every } v\in\alpha,w\in\beta},$$
and that every edge of $\sd(\gr G)$ is of the form $v\alpha$, for some $v\in\alpha\in E$. Next, denote
$$P=P_f(\gr G)=\set{(v\alpha,w\beta)\in E\Par{\sd(\gr G)}^2\ :\ v\neq w \text{ and } vw\notin E},$$
and define the \emph{fine string obstruction} of $\gr G$ as
$$\vk o_{\gr G}^f=\vk o_{\sd(\gr G)}(P).$$
\end{df}}

\begin{df}
Let $\gr G=(V,E)$ be a graph, and recall that $\gr G^*$ denotes the barycentric subdivision of $\gr G$: the edges of $\gr G^*$ are $v\alpha$ where $v\in\alpha\in E$. Denote
$$P=P_{\sd}(\gr G)=\set{(v\alpha,w\beta)\in E\Par{\gr G^*}^2\ :\ v\neq w \text{ and } vw\notin E},$$
and define the \emph{subdivided string obstruction} of $\gr G$ as
$$\vk o_{\gr G}^{\sd}=\vk o_{\gr G^*}(P).$$
\end{df}
Note that $\vk o_{\gr G}^{\sd}$ is in fact a Van Kampen obstruction of $\gr G^*$ modified for $P$ (in the sense of Definition \ref{vk_def}). We associate it with $\gr G$, however, because it vanishes when $\gr G$ is a string graph. In fact, $P_{\sd}(\gr G)$ contains precisely the pairs of edges whose intersections are forbidden if $\gr G$ is a string graph, in the following sense:

\begin{thm} \label{ob_2_eq}
Let $\gr G=(V,E)$ be a graph. Then $\gr G$ is a string graph if and only if there exists a drawing of the barycentric subdivision $g:\gr G^*\to\Rt$, such that
$$g(v\alpha)\cap g(w\beta)=\emp$$
for every $(v\alpha,w\beta)\in P_{\sd}(\gr G)$.
\end{thm}

\begin{proof}
Suppose $g:\gr G^*\to\Rt$ is as above. For every $v\in E$ let $\gr T_v$ denote the subgraph of $\gr G^*$ induced on the collection (of $\gr G^*$-vertices) $\set{v}\cup\set{\alpha\in E\ :\ v\in\alpha}$. From our assumption on the properties of $g$, it is clear that for every $v,w\in V$ with $v\neq w$ we have $g(\gr T_v)\cap g(\gr T_w)\neq\emp$ if and only if $vw\in E$. For every $v\in V$ we know that $g(\gr T_v)$ is connected (in fact $\gr T_v$ is a tree), so let $\gamma_v=g(\gr T_v)$ be a path tracing along the line segments $\bigcup_{\alpha\ni v}g(v\alpha)$. This gives us a representation of $\gr G$ as a string graph.

For the converse, suppose that $\gr G$ is a string graph represented by a collection of paths $\gamma_v\su\Rt$ for every $v\in V$. So for every $v\neq w$ in $V$, the intersection $\gamma_v\cap\gamma_w$ is non-empty if and only if $vw\in E$. For every $v\in V$ define $g(v)$ as an arbitrary point in $\gamma_v$, and for every $\alpha=vw\in E$ define $g(\alpha)$ as an arbitrary point in $\gamma_v\cap\gamma_w$. For every $v\in\alpha\in E$, note that both $g(v)$ and $g(\alpha)$ are in $\gamma_v$. We may therefore connect them within this image, and obtain $g:\gr G^*\to\Rt$ such that $g(v\alpha)\su\gamma_v$ for every $v\in\alpha\in E$. For every $(v\alpha,w\beta)\in P_{\sd}(\gr G)$, we clearly have $vw\notin E$, therefore $g(v\alpha)\cap g(w\beta)\su \gamma_v\cap\gamma_w=\emp$ completing our proof.
\end{proof}

It might seem that we have described a second obstruction for string graphs. However, we will now prove in the following lemma that they vanish together, therefore once we know $\vk o_{\gr G}^s$, computing $\vk o_{\gr G}^{\sd}$ tells us nothing new about whether $\gr G$ is a string graph or not. Half of the statement is also our final piece for the proof of Theorem \ref{ob_1_van}.

\begin{lem} \label{ob_eq}
Let $\gr G=(V,E)$ be a graph. Then $\vk o_{\gr G}^s$ vanishes if and only if $\vk o_{\gr G}^{\sd}$ vanishes.
\end{lem}

\begin{proof}
We will prove that there exists a function $f:\gr G\to\Rt$ such that $\abs{f(\alpha)\cap f(\beta)}$ is even for every $(\alpha,\beta)\in P_s(\gr G)$, if and only if there exists a function $g:\gr G^*\to\Rt$ such that $\abs{g(v\alpha)\cap g(w\beta)}$ is even for every $(v\alpha,w\beta)\in P_{\sd}(\gr G)$. This statement is equivalent to our lemma, due to Proposition \ref{vk_prop} (applied for both obstructions).

Suppose such $g:\gr G^*\to\Rt$ exists. Let $f:\gr G\to\Rt$ denote the composition of the homeomorphism $\gr G\to\gr G^*$ with $g$. Let $(\alpha,\beta)\in P_s(\gr G)$, where $\alpha=v_1 v_2$ and $\beta =w_1 w_2$. As line segments, we have $\alpha=v_1\alpha\cup v_2\alpha$ and $\beta=w_1\beta\cup w_2\beta$. Clearly we have $(v_i\alpha,w_j\beta)\in P_{\sd}(\gr G)$ for every $i,j=1,2$. We deduce that 
\begin{equation} \label{subd_sum}
\abs{f(\alpha)\cap f(\beta)}=\sum_{i,j=1}^2 \abs{g(v_i\alpha)\cap g(w_j\beta)}
\end{equation}
is even as a sum of even terms.\\
Conversely, if $f:\gr G\to\Rt$ exists, let $g:\gr G^*\to\Rt$ be obtained from $f$ by identifying $\gr G$ and $\gr G^*$ as homeomorphic, taking care that for every $\alpha\in E$, the point $g(\alpha)$ is chosen as a point of $f(\alpha)$ where $f(\alpha)$ doesn't cross itself or any other edges. Suppose that $\abs{g(v\alpha)\cap g(w\beta)}$ is odd, for some $v\in\alpha\in E$ and $w\in\beta\in E$ where $(v\alpha,w\beta)\in P_{\sd}(\gr G)$. We say that the pair $(\alpha,\beta)$ is \emph{problematic}, and denote $\alpha=v_1 v_2$ and $\beta=w_1 w_2$. For every such problematic pair, we will modify $g$ using the finger moves (for $\gr G^*$):
$$\phi_{v_1\alpha,\beta}\ ,\ \phi_{v_2\alpha,\beta}\ ,\ \phi_{w_1\beta,\alpha}\ ,\ \phi_{w_2\beta,\alpha}$$
to obtain some modified $g':\gr G^*\to\Rt$. Note that these finger moves only effect the parity of $\abs{g(v_i\alpha)\cap g(w_j\beta)}$ for $i,j=1,2$, and applying some or all of these finger moves allows us to freely move any number of intersection points between the four sets $g(v_i\alpha)\cap g(w_j\beta)$ ($i,j=1,2$). We proceed as follows: if $v_k w_l\in E$ for some $k,l\in\set{1,2}$, then $(v_k\alpha,w_l\beta)\notin P_{\sd}(\gr G)$, and we may apply the finger moves to guarantee that $\abs{g(v_i\alpha)\cap g(w_j\beta)}$ is even for every $(i,j)\neq(k,l)$. If $v_i w_j\notin E$ for every $i,j=1,2$ then $(\alpha,\beta)\in P_s(\gr G)$. Clearly (\ref{subd_sum}) still holds, and the sum is even by our assumption on $f$. We may therefore apply the finger moves to guarantee that $\abs{g(v_i\alpha)\cap g(w_j\beta)}$ is even for every $i,j=1,2$. Proceeding similarly for every problematic pair $(\alpha,\beta)$, we obtain our modified $g'$ as required.
\end{proof}

\begin{proof}[Proof of Theorem \ref{ob_1_van}]
Let $\gr G$ be a string graph. By Theorem \ref{ob_2_eq}, and Proposition \ref{vk_prop}, $\vk o_{\gr G}^{\sd}$ vanishes. Our proof now follows from Lemma \ref{ob_eq}.
\end{proof}

\section{Final remarks} 
In an attempt to strengthen Theorem \ref{ob_1_van}, we may consider an \emph{integer valued} modified Van Kampen obstruction for string graphs. This can be achieved using cohomology with integer coefficients, or as a variation of the geometric definitions we used. Given $f:\gr G\to\Rt$, every intersection point in $f(\alpha)\cap f(\beta)$ (where $\alpha,\beta\in E(\gr G)$) has either a positive or negative sign. Defining $\widetilde o_f(\alpha,\beta)$ as the sum of these contributions gives us an element $\widetilde o_f\in\bbZ^P$ such that $o_f$ is the modulo $2$ reduction of $\widetilde o_f$. This was already considered (in a different form) by Tutte \cite{Tut}, see \cite[Appendix D]{MTW} for the integer Van Kampen obstruction. Clearly there are functions $f$ such that $\widetilde o_f$ is non-zero while $o_f=0\in\Zt^P$. However, the author is currently unaware of any example $\gr G$ for which $\vk o_{\gr G}^s$ vanishes but the integer valued obstruction does not ($\widetilde o_f\neq 0\in\bbZ^{P_s(\gr G)}$ for every $f$).

Our final remark relates to the computational complexity of recognizing string graphs. Determining if $\vk o_{\gr G}^s$ vanishes can be done in polynomial time (even if we consider the integer valued obstruction), however Kratochv\'il \cite{Kra2} proved that recognizing string graphs is NP-hard (later determined to be NP-complete \cite{SSS}). Proof of NP-hardness is done by reducing the planar 3-satisfiability problem to determining whether certain graphs are string graphs. Avoiding almost all details, a formula $\phi$ (in conjunctive normal form) with certain properties is satisfiable if and only if some graph $\gr G$ derived from $\phi$ is a string graph. String representations of $\gr G$ are in one-to-one correspondence with variable assignments that satisfy $\phi$.\\
Obtaining a string graph representation for $\gr G$ from an assignment which doesn't satisfy $\phi$ would be possible if $C*\overline C_6$ (see Figure \ref{Kratochvil_figure}) were a string graph: representing each subgraph of $\gr G$ corresponding to a false clause in $\phi$ in a manner consistent with the assignment, is equivalent to obtaining a string representation of $C*\overline C_6$. Since we have $f:C*\overline C_6\to\Rt$ for which $o_f=0$, any arbitrary assignment to the variables of $\phi$ indicates that $\vk o_{\gr G}^s$ vanishes, whether $\phi$ is satisfiable or not. We thus deduce, perhaps unsurprisingly, that the computationally easy task of determining whether the string obstruction vanishes is of no use to the NP-hard problem described by Kratochv\'il.

\section*{Acknowledgments} 
This work was completed as part of the PhD thesis of the author. I would like to thank my advisor Gil Kalai, as well as Jan Kratochv\'il and Geva Yashfe.

\bibliographystyle{plain}

\end{document}